\documentclass[12pt]{amsart}

\usepackage{pinlabel}
\usepackage{amsmath,amsfonts,amssymb,latexsym,mathrsfs,stmaryrd}
\usepackage{color}
\usepackage{graphicx}
\usepackage{graphics}
\usepackage{enumerate}
\usepackage{amscd} 
\usepackage{amsxtra}
\usepackage{epic,eepic}
\usepackage{hyperref} 
\usepackage{enumitem}
\usepackage{MnSymbol}
\usepackage{dsfont}

\newtheorem{theorem}{Theorem}[section]

\newtheorem{proposition}[theorem]{Proposition}

\newtheorem{lemma}[theorem]{Lemma}

\theoremstyle{definition} 
\newtheorem{defn}[theorem]{Definition}
\newtheorem{definition}[theorem]{Definition}

\newtheorem{remark}[theorem]{Remark}

\newcommand{\cC}{\mathcal{C}} 

\newcommand{\cP}{\mathcal{P}}

\newcommand{\qu}{/\kern-.7ex/}
\newcommand{\lqu}{\backslash \kern-.7ex \backslash}
\newcommand{\on}{\operatorname}

\newcommand{\Pic}{\on{Pic}}

\title{$K$-theoretic quasimap invariants and their wall-crossing}

\author{Hsian-Hua Tseng}
\address{Department of Mathematics\\ Ohio State University\\ 100 Math Tower, 231 West 18th Ave.\\Columbus\\ OH 43210\\ USA}
\email{hhtseng@math.ohio-state.edu}

\author{Fenglong You}
\address{Department of Mathematics\\ Ohio State University\\ 100 Math Tower, 231 West 18th Ave.\\Columbus\\ OH 43210\\ USA}
\email{you.111@osu.edu}


\begin{document}
\date{\today}

\begin{abstract} 
For each positive rational number $\epsilon$, we define $K$-theoretic $\epsilon$-stable quasimaps to certain GIT quotients $W\sslash G$. For $\epsilon>1$, this recovers the $K$-theoretic Gromov-Witten theory of $W\sslash G$ introduced in more general context by Givental and Y.-P. Lee. 

For arbitrary $\epsilon_1$ and $\epsilon_2$ in different stability chambers, these $K$-theoretic quasimap invariants are expected to be related by wall-crossing formulas. We prove wall-crossing formulas for genus zero $K$-theoretic quasimap theory when the target $W\sslash G$  admits a torus action with isolated fixed points and isolated one-dimensional orbits.
\end{abstract}

\maketitle 

\tableofcontents

\section{Introduction}

For an affine algebraic variety $W=\on{Spec}(A)$ that admits an action by a reductive algebraic group $G$, a choice of the polarization $\mathcal O(\theta)$ determines a GIT quotient $W\sslash_\theta G$. In \cite{CKM} and \cite{CK}, the authors define the moduli space $$Q^\epsilon_{g,k}(W\sslash G,d)$$ parametrizing maps of class $d$ from genus $g$ nodal curves with $k$-marked points to the quotient stack $[W\slash G]$ with $\epsilon$-stability. Assuming $W$ has only lci singularity, the canonical obstruction theory of $Q^\epsilon_{g,k}(W\sslash G,d)$ is perfect and hence yields a virtual fundamental class, see \cite{CKM} and \cite{CK}. 
When $W\sslash G$ is projective the cohomological quasimap invairants are defined in \cite{CKM} and \cite{CK}, using evaluation maps and descendant classes $\psi$ at the markings.

By \cite[Section 2.3]{Lee}, the perfect obstruction theory yields a {\em virtual structure sheaf} $$\mathcal O^{\text{vir}}_{Q^\epsilon_{g,k}(W\sslash G,d)}$$ in the $K$-theory of $Q^\epsilon_{g,k}(W\sslash G,d)$. $K$-theoretic (descendant) $\epsilon$-quasimap invariants of $W\sslash G$ are defined to be holomorphic Euler characteristics of vector bundles on the moduli space $Q^\epsilon_{g,k}(W\sslash G,d)$:
\begin{align*}
&\langle \gamma_1 L_1^{a_1},\ldots,\gamma_kL_k^{a_k}\rangle_{g,k,d}^{W\sslash G, \epsilon}\\
:=
&\chi
\left(
Q^\epsilon_{g,k}(W\sslash G,d), \mathcal O^{\text{vir}}_{Q^\epsilon_{g,k}(W\sslash G,d)}\otimes (\otimes_{i=1}^kL_i^{\otimes a_i}\on{ev}_i^*(\gamma_i))
\right)
\end{align*}
where $a_i$ are nonnegative integers, $\gamma_i\in K^0(W\sslash G)$ and $L_i$ are tautological line bundles over $Q^\epsilon_{g,k}(W\sslash G,d)$ corresponding to the $i$-th marked points.

As explained in \cite{CKM} and \cite{CK}, for each fixed class $d$, the set of positive rational numbers can be divided into chambers by finitely many walls $1,\frac{1}{2},\cdots,\frac{1}{d(L_\theta)}$, such that the moduli space $Q^\epsilon_{g,k}(W\sslash G,d)$ stays constant when $\epsilon$ is changing within a chamber, where $d(L_\theta)$ may be considered as the degree of the map with respect to the polarization $\mathcal O(\theta)$. We write $\epsilon=0+$ for $\epsilon$ being sufficiently small and being in the first chamber $(0,\frac{1}{d(L_\theta)}]$, for all $d$. Changes of quasimap invariants as $\epsilon$ varies,  termed {\em wall-crossing formulas}, is proved in \cite{CK} for genus $0$ equivariant cohomological theory. 

The goal of this paper is to study wall-crossing behavior for $K$-theoretic genus $0$ quasimap theory. 

We will consider permutation equivariant version of quantum $K$-theory, which takes into account the $S_n$-action on the moduli space by permuting the marked points, developed by Givental \cite{Givental}. This permutation-equivariant theory works better in our context.

As in \cite{CK}, genus zero wall-crossing formulas are naturally stated via generating functions of $K$-theoretic quasimap invariants. Let $\{\phi_\alpha\}$ be a basis of $K^0(W\sslash G)\otimes \mathbb Q$, $\{\phi^\alpha\}$ be the dual basis and $t=\sum_\alpha t^\alpha \phi_\alpha\in K^0(W\sslash G)\otimes \mathbb Q$. Let $\Lambda$ be the $\lambda$-algebra described in Section \ref{section-permut-equiv} and $\gamma\in K^0(W\sslash G,\Lambda)$. For $\epsilon \geq 0+$, we define the $S$-operator:
\[
(S^\epsilon)(q)(\gamma)=\gamma+\sum\limits_\alpha\left(\sum\limits_{(n,d)\neq (0,0)}Q^d\langle \frac{\phi^\alpha}{1-qL},\gamma,t,\ldots,t\rangle^{\epsilon,S_n}_{0,n+2,d}\right)\phi_\alpha,
\]
where $q$ is a formal variable. 

Suppose that $W\sslash G$ admits a torus $T$ action with  isolated fixed points and isolated one-dimensional orbits. The main result of this paper, Theorem \ref{main-theorem-1}, is a wall-crossing formula which relates $S$-operators for $\epsilon_1$ and $\epsilon_2$ in different stability chambers and the invertible classes $\gamma$ of the form $\gamma=\mathds 1+O(Q)$.

For each $\epsilon\geq 0+$, we also defined the permutation-equivariant $\mathcal J^\epsilon$-function and prove the following identity
\[
\mathcal J^\epsilon(q)=S^\epsilon(q)(P^\epsilon)
\]
where $P^\epsilon$ is a generating series on the quasimap graph space, see Proposition \ref{prop:SP}. Theorem \ref{main-theorem-2}  below shows, for each $\epsilon\geq 0+$, the permutation-equivariant $\mathcal J^\epsilon$-function lies in the Lagrangian cone $\mathcal L_{S_{\infty},W\sslash G}$ of the permutation-equivariant $K$-theoretic Gromov-Witten theory of $W\sslash G$. 

Theorem \ref{main-theorem-1} and Theorem \ref{main-theorem-2} generalize the main theorems of \cite{CK} to $K$-theory when the torus action on $W\sslash G$ has only isolated fixed points and isolated one-dimensional orbits. They can also be considered as generalizations of the $K$-theoretic mirror theorem for toric manifolds due to Givental \cite{Givental}.

\subsection{Acknowledgments}
H.-H. T. is supported in part by Simons Foundation Collaboration Grant and NSF grant DMS-1506551.

\section{Constructions}

\subsection{$K$-theoretic Quasimap Invariants}
In this section, we review some basic definitions in quasimap theory following \cite{CKM, CK}. We then define the $K$-theoretic version of quasimap invariants.

Consider an affine algebraic variety $W=\on{Spec} (A)$ with $G$-action, where $G$ is a reductive algebraic group. An element $\xi$ of the character group $\chi(G)$, determines a one-dimensional $G$-representation $\mathbb C_\xi$, and hence an element 
\[
L_{\xi}=W\times \mathbb C_{\xi}
\]
of the group $\Pic_G(W)$ of isomorphism classes of $G$-linearized line bundles on $W$.

Fixing a character $\theta\in \chi(G)$, we obtain the GIT quotient $W\sslash G:=W\sslash _\theta G$, which is a quasiprojective variety and the morphism 
\[
W\sslash G\rightarrow W\slash _{\text {aff}}G:=\on{Spec}(A^G)
\]
is a projective morphism.

We write $W^s=W^{s}(\theta)$ (respectively $W^{ss}=W^{ss}(\theta)$) for the stable (respectively semistable) determined by $\theta$. Following \cite{CKM} and \cite{CK}, we require the following assumptions for the rest of the paper:
\begin{itemize}
\item $W^s=W^{ss}\neq \emptyset$;
\item $W^s$ is nonsingular;
\item $G$ acts freely on $W^s$.
\end{itemize}

Let $(C,x_1,\ldots,x_k)$ be a prestable $k$-pointed curve, a map
\[
[u]:C\rightarrow [W/G]
\]
is represented by a pair $(P,u)$, where $P\rightarrow C$ is a principal $G$-bundle over $C$ and 
\[
u:C\rightarrow P\times _G W
\]
is a section of the bundle $\rho: P\times_G W\rightarrow C$.

The numerical class $d$ of a map $(P,u)$ is the group homomorphism
\[
d: \Pic_G(W)\rightarrow \mathbb Z, \quad L\mapsto \deg_C(u^*(P\times_G L))
\]
\begin{defn}[\cite{CK}, Definition 2.4.1]
A quasimap to $W\sslash G$ is a map from $((C,x_1,\ldots,x_k),P,u)$ to the quotient stack $[W/G]$ such that generic points of $C$ land on the stable locus of $W$, i.e. for a generic point $p$ of $C$, we have $u(p)\in P\times _G W^s$.
\end{defn}
\begin{remark}
Points on $C$ which map to the unstable locus of $W$ are called base points, hence a quasimap has at most finitely many base points.
\end{remark}

\begin{defn}[\cite{CK}, Definition 2.4.2]
A group homomorphism $d: \Pic_G(W)\rightarrow \mathbb Z$ is called $L_\theta$-effective if it is a finite sum of classes of quasimaps, we write $\on{Eff}(W,G,\theta)$ for the semigroup of $L_\theta$-effective classes.
\end{defn}

Let $\on{Qmap}_{g,k}(W\sslash G,d)$ be the moduli space of genus $g$, $k$-pointed quasimaps of class $d$ to $W\sslash G$.

\begin{defn}[\cite{CK}, Definition 2.4.4]
We say that a quasimap is prestable if its base points are away from the nodes and marked points of the underlying curve. 
\end{defn}

\begin{defn}[\cite{CK}, Definition 2.4.5]
Given a prestable quasimap 
\[
((C,x_1,\ldots,x_k),P,u)
\] 
to $W\sslash G$, the length $l(x)$ at a point $x\in C$ is the contact order of $u(C)$ with the unstable subscheme $P\times _G W^{us}$ at $u(x)$. More precisely, 
\[
l(x):=\text{length}_x(\on{coker}(u^*\mathcal J \rightarrow \mathcal O_C)),
\]
 where $\mathcal J$ is the ideal sheaf of the closed subscheme $P\times _G W^{us}$ of $P\times _G W$.
\end{defn}
\begin{defn}[\cite{CK}, Definition 2.4.6]
Given a positive rational number $\epsilon$, a quasimap 
\[
((C,x_1,\ldots,x_k),P,u)
\] 
to $W\sslash G$ is called $\epsilon$-stable if it is prestable and satisfies the following conditions
\begin{itemize}
\item $\omega_C(\sum_{i=1}^k x_i)\otimes \mathcal L_\theta^\epsilon$ is ample, where $\mathcal L_\theta=P\times _G \mathbb C_\theta=u^*(P\times_G L_\theta)$;
\item $\epsilon l(x)\leq 1$ for every point $x \in C$.
\end{itemize}
\end{defn}
The moduli space of $\epsilon$-stable quasimaps $Q^\epsilon_{g,k}(W\sslash G,d)$ is an open substack of $\on{Qmap}_{g,k}(W\sslash G,d)$. The universal family 
\[
((\cC^\epsilon, x_1,\ldots,x_k),\cP,u)
\]
over $Q^\epsilon_{g,k}(W\sslash G,d)$ is obtained as follows. Let $\mathfrak M_{g,k}$ be the moduli space of prestable curves and $\mathfrak{Bun}_G$ be the moduli stack of principal $G$-bundles on the fibers of the universal curve $\pi: \mathfrak C_{g,k}\rightarrow \mathfrak M_{g,k}$. The universal curve 
\[
\pi: \cC^\epsilon \rightarrow Q^\epsilon_{g,k}(W\sslash G,d)
\]
is the pull-back of $\mathfrak C_{g,k}$ via the natural forgetful morphism
\[
\mu: Q^\epsilon_{g,k}(W\sslash G,d)\rightarrow \mathfrak{M}_{g,k}.
\]
 $\cP$ is the pull-back of the universal curve over $\mathfrak{Bun}_G$ via the natural forgetful morphism
\[
\nu: Q^\epsilon_{g,k}(W\sslash G,d)\rightarrow \mathfrak{Bun}_G,
\]
and $u$ is a section of the bundle
\[
\rho: \cP\times_G W\rightarrow \cC^\epsilon
\]

We write $\mathbb R T_\rho$ for the relative tangent complex of $\rho$. The canonical obstruction theory of $Q^\epsilon_{g,k}(W\sslash G,d)$ relative to the smooth Artin stack $\mathfrak{Bun}_G$ is given by the complex
\begin{equation}\label{Obst-theory}
(R^\bullet \pi_*(u^*\mathbb R T_\rho))^\vee
\end{equation}

\begin{theorem}[\cite{CKM}, Theorem 7.1.6 \& \cite{CK}, Theorem 2.5.1]
If $W$ has only lci singularities, then the obstruction theory \eqref{Obst-theory} is perfect.
\end{theorem}
Hence, when $W$ has lci singularities, by \cite[Section 2.3]{Lee}, there is a virtual structure sheaf $\mathcal O^{\text{vir}}_{Q^\epsilon_{g,k}(W\sslash G,d)}$, which is an element of $K(Q^\epsilon_{g,k}(W\sslash G,d))$, the Grothendieck group of coherent sheaves on $Q^\epsilon_{g,k}(W\sslash G,d)$.

Since base points are away from the marking, we have the evaluation maps
\[
\on{ev}_i: Q^\epsilon_{g,k}(W\sslash G,d)\rightarrow W\sslash G, \quad i=1,\ldots,k.
\]
We assume that $W\sslash G$ is projective. Let $L_i$ be the $i$-th tautological cotangent line bundle over $Q^\epsilon_{g,k}(W\sslash G,d)$.

\begin{definition}
Given $\gamma_i\in K^0(W\sslash G)\otimes \mathbb Q$ and nonnegative integers $a_i$, $1\leq i \leq k$, we define the $K$-theoretic descendant $\epsilon$-stable quasimap invariant of $W\sslash G$:
\begin{align}\label{quasimap-invariant}
\langle \gamma_1 L_1^{a_1},\ldots,\gamma_kL_k^{a_k}\rangle_{g,k,d}^{W\sslash G, \epsilon}:=
\chi
\left(
Q^\epsilon_{g,k}(W\sslash G,d), \mathcal O^{\text{vir}}_{Q^\epsilon_{g,k}(W\sslash G,d)}\otimes (\otimes_{i=1}^kL_i^{\otimes a_i}\on{ev}_i^*(\gamma_i))
\right)
\end{align}
\end{definition}

\begin{remark}[see \cite{CK}, Remark 2.4.8]
\hfill
\begin{enumerate}
\item
For $\epsilon>1$, assume $(g,k)\neq (0,0)$, we have 
\[
Q^{\epsilon}_{g,k}(W\sslash G,d)=\overline{M}_{g,k}(W\sslash G,d)
\]
and the corresponding $K$-theoretic $\epsilon$-stable quasimap invariant is a $K$-theoretic Gromov-Witten invariant as defined by A. Givental and Y.-P. Lee \cite{GL} and \cite{Lee}

\item

Assume that $(g,k)\neq (0,0),(0,1)$ and fixed the numerial data $(g,k,d)$, for each integer $1\leq e\leq d(L_\theta)-1$, the moduli space $Q^{\epsilon}_{g,k}(W\sslash G,d)$ of $\epsilon$-stable quasimaps to $W\sslash G$ stays constant when $1/(e+1)<\epsilon \leq 1/e$. Therefore, for each fixed $d$, the set of positive rational number is divided into chambers by finitely many walls $1,\frac12,\cdots, \frac{1}{d(L_\theta)}$.
\end{enumerate}
\end{remark}

As explained in \cite{CK}, the quasimap theory applies to a large class of targets, including toric and flag varieties, zero loci of sections of homogeneous bundles on toric and flag varieties, local targets with base a GIT quotient, Nakajima quiver varieties etc.

\subsection{Quasimap Graph Spaces}
We write $QG^\epsilon_{g,k}(W\sslash G,d)$ for $\epsilon$-stable quasimap graph spaces. The data
\[
((C,x_1,\ldots,x_k),P,u,\varphi)
\]
represents the maps
\[
C\longrightarrow [W/G]\times \mathbb P^1
\]
of class $(d,1)$ and the morphism $\varphi$ maps from $C$ to $\mathbb P^1$ with degree $1$. Namely, there is an irreducible component $C_0$ of $C$ such that the restriction $\varphi|_{C_0}\rightarrow \mathbb P^1$ is an isomorphism and the remaining components $C\backslash C_0$ are contracted by $\varphi$. Elements of $QG^\epsilon_{g,k}(W\sslash G,d)$ is given by the prestable quasimaps $((C,x_1,\ldots,x_k),P,u,\varphi)$ with the stability conditions:
\begin{itemize}
\item $\omega_{\overline{C\backslash C_0}}(\sum x_i+\sum y_j)\otimes \mathcal L_\theta^\epsilon$ is ample, where $x_i$ are marked points on $\overline{C\backslash C_0}$ and $y_j$ are the nodes $\overline{C\backslash C_0}\cap C_0$;
\item the inequality $\epsilon l(x)\leq 1$ holds for every point $x$ on $C$. 
\end{itemize}
The $K$-theoretic $\epsilon$-stable quasimap invariants for graph spaces can be defined the same way as \eqref{quasimap-invariant}.

\subsection{Permutation Equivariant Quantum $K$-theory}\label{section-permut-equiv}

In this section we review the permutation equivariant $K$-theoretic Gromov-Witten theory developed by Givental in \cite{Givental}. 

By a $\lambda$-algebra $\Lambda$, we mean an algebra over $\mathbb Q$ equipped with abstract Adams operations $\Psi^k$, $k=1,2,\ldots, $ that is,
ring homomorphisms $\Lambda \rightarrow \Lambda$ that satisfy $\Psi^r\Psi^s=\Psi^{rs}$ and $\Psi^1=\text{id}$.
We assume that the $\lambda$-algebra $\Lambda$ includes the Novikov variables, the algebra of symmetric polynomials in a given number of variables and the torus equivariant $K$-ring of a point. We also assume $\Lambda$ has a maximal ideal $\Lambda_+$ with the corresponding $\Lambda_+$-adic topology. We write $\mathcal K$ for the space of rational functions of $q$ with coefficients from $K^0(X)\otimes \Lambda$, the space $\mathcal K$ is equipped with a symplectic form
\[
\Omega(f,g):=-[\on{Res}_{q=0}+\on{Res}_{q=\infty}](f(q^{-1}),g(q))\frac{dq}{q}.
\]
where $(\cdot,\cdot)$ is the $K$-theoretic intersection pairing on $K^0(X)\otimes \Lambda$:
\[
(a,b):=\chi(X;a\otimes b).
\]
Then $\mathcal K$ is a symplectic linear space. It can be decomposed into the direct sum 
\[
\mathcal K=\mathcal K_+\oplus \mathcal K_-
\]
where $\mathcal K_+$ is the subspace of Laurent polynomials in $q$ and $\mathcal K_-$ is the complementary subspace of rational functions of $q$ regular at $q=0$ and vanishing at $q=\infty$.

Given a compact K\"ahler manifold $X$, consider the $S_n$ modules
\[
[t(L),\ldots, t(L)]:=\sum_m (-1)^m H^m(\overline{M}_{0,n}(X,d); \mathcal O_{\overline{M}_{0,n}(X,d)}^{\text vir}\otimes _{i=1}^n t(L_i)),
\]
Where the input $t(q)$ is a Laurent polynomial in $q$ with coefficients in $K^0(X)\otimes \Lambda$ and $\Lambda$ is a $\lambda$-algebra.

The correlators of permutation equivariant quantum $K$-theory is defined as
\[
\langle t(L),\ldots, t(L)\rangle_{0,n,d}^{S_n}:=\pi_*(\mathcal O_{\overline{M}_{0,n}(X,d)}^{\text vir}\otimes _{i=1}^n t(L_i)),
\]
where $\pi_*$ is the $K$-theoretic push forward along the projection 
\[
\pi: \overline{M}_{0,n}(X,d)/S_n\rightarrow [pt].
\]
The permutation equivariant $K$-theoretic quasimap invariants can be defined similarly, replacing moduli spaces of stable maps by moduli spaces of stable  quasimaps. When $X$ admits a torus action, the extension of $K$-theoretic quasimap theory to torus equivariant setting is also straightforward.

\subsection{The $J^\epsilon$ Functions}

Consider the $K$-theoretic Poincar\'e metric 
\[
g_{\alpha\beta}=\langle \phi_\alpha,\phi_\beta\rangle:=\chi(W\sslash G,\phi_\alpha\phi_\beta),
\]
where $\{\phi_\alpha\}$ is a basis of $K^0(W\sslash G)\otimes \mathbb Q$, 
we define permutation-equivariant $K$-theoretic $J^\epsilon$-function of $X=W\sslash G$ as
\begin{equation}\label{J^epsilon-function}
\mathcal J^\epsilon_X(t,q):=\mathds 1+\frac{t}{1-q}+\sum_{\alpha,\beta}\sum_{(n,d)\neq (0,0),(1,0)}Q^d\langle \frac{\phi_\alpha}{(1-q)(1-qL)},t,\ldots, t\rangle_{0,n+1,d}^{\epsilon,S_n} g^{\alpha\beta}\phi_\beta.
\end{equation}
where $t=\sum_\alpha t^\alpha\phi_\alpha\in K^0(W\sslash G)\otimes \mathbb Q$ and the unstable terms in the summation, that is, when $n=0,d\neq 0, d(L_\theta)\leq 1/\epsilon$, are defined the same way as in \cite[Definition 5.1.1]{CK}, via $\mathbb C^*$-localization on graph spaces: Choose coordinates $[x_0,x_1]$ on $\mathbb P^1$, then the standard $\mathbb C^*$-action on $\mathbb P^1$ is
\[
t[x_0,x_1]=[tx_0,x_1], \quad \forall t\in \mathbb C^*.
\]
This action naturally induces an action on the $\epsilon$-stable quasimap graph spaces  $QG^\epsilon_{g,k}(W\sslash G,d)$. For $k=0$ and $\epsilon \leq \frac{1}{d(L_\theta)}$, we write 
\[
F_0\cong Q_{0,1}(W\sslash G,d)_0
\]
for the fixed locus parametrizing  quasimaps of class $d$ of the form $(C=\mathbb P^1,P,u)$ with a principal $G$-bundle $P$ on $\mathbb P^1$, a section $u:\mathbb P^1\rightarrow P\times_G W$ such that $u(x)\in P\times_G W^s$ for $x\neq 0\in \mathbb P^1$ and $0\in \mathbb P^1$ is a base-point of length $d(L_\theta)$.
Then the unstable terms in \eqref{J^epsilon-function} is defined as
\[
\sum_{\alpha,\beta}\sum\limits_{d\neq 0,d(L_\theta)\leq 1/\epsilon}Q^d\chi\left(F_0,\mathcal O^{\text{vir}}_{F_0}\otimes\left(\frac{ev_1^*(\phi_\alpha)}{\on{tr}_{\mathbb C^*}\wedge^*N^*_{F_0}}\right)\right)g^{\alpha\beta}\phi_\beta,
\]
where $N^*_{F_0}$ is the conormal bundle of the fixed locus $F_0$.

The permutation-equivariant big $J$-function of $X$ is 
\[
\mathcal J_X(t(q),q):=\mathds 1+\frac{t(q)}{1-q}+\sum_{\alpha,\beta}\sum_{(n,d)\neq (0,0),(1,0)}Q^d\langle \frac{\phi_\alpha}{(1-q)(1-qL)},t(L),\ldots, t(L)\rangle_{0,n+1,d}^{S_n}g^{\alpha\beta} \phi_\beta, 
\]
where $t(q)$ is a Laurent polynomial in $q$ with coefficients in $K^0(X)\otimes \Lambda$. We write $\mathcal L_{S_\infty,X}$ for the range of the big $J$-function in permutation-equivariant quantum $K$-theory of $X$.



\section{$S$ \& $P$}

\subsection{The $S$-operator}
We use double brackets to denote the generating function
\[
\llangle\gamma_1L_1^{a_1},\ldots,\gamma_kL_k^{a_k}\rrangle_{0,k}^\epsilon:=\sum\limits_{n\geq 0,d\geq 0}Q^d\langle \gamma_1L_1^{a_1},\ldots,\gamma_kL_k^{a_k},t,\ldots,t\rangle_{0,k+n,d}^{\epsilon,S_n}
\]
where $t\in K^0(W\sslash G)\otimes \Lambda$ and summing over terms when $Q^\epsilon_{0,k+n}(W\sslash G,d)/S_n$ exists (does not include unstable terms). 

We consider the permutation-equivariant $\mathcal J^\epsilon$-function of $X=W\sslash G$,
\[
\mathcal J^\epsilon_X(t,q):=\mathds 1+\frac{t}{1-q}+\sum_{\alpha,\beta}\sum_{(n,d)\neq (0,0),(1,0)}Q^d\langle \frac{\phi_\alpha}{(1-q)(1-qL)},t,\ldots, t\rangle_{0,n+1,d}^{\epsilon,S_n}g^{\alpha\beta}\phi_\beta.
\]
with unstable terms defined as before and $t\in K^0(W\sslash G)\otimes \mathbb Q$.
Givental \cite{Givental 00} introduced a non-constant metric for permutation-equivariant quantum $K$-theory. Similarly, we can define a non-constant metric for each $\epsilon\geq 0+$:
\[
G_{\alpha\beta}^\epsilon:=g_{\alpha\beta}+\llangle \phi_\alpha,\phi_\beta\rrangle_{0,2}^\epsilon,
\]
and the inverse tensor
\begin{align*}
G^{\alpha\beta}_\epsilon=& g^{\alpha\beta}-\llangle \phi^\alpha,\phi^\beta\rrangle_{0,2}^\epsilon+\sum\limits_\mu\llangle \phi^\alpha,\phi^\mu\rrangle_{0,2}^\epsilon\llangle \phi_\mu,\phi^\beta\rrangle_{0,2}^\epsilon\\
& -\sum\limits_{\mu,\nu}\llangle \phi^\alpha,\phi^\mu\rrangle_{0,2}^\epsilon\llangle\phi_\mu,\phi^\nu\rrangle_{0,2}^\epsilon\llangle\phi_\nu,\phi^\beta\rrangle_{0,2}^\epsilon+\cdots
\end{align*}
where $\{\phi^\alpha\}$ is the Poincar\'e  dual basis of $\{\phi_\alpha\}$.

The operator $S^\epsilon_t:\mathcal K\rightarrow \mathcal K$ is defined as
\[
(S^\epsilon_t)(q)(\gamma)=\sum\limits_{\alpha,\beta}\left(\langle \phi_\alpha,\gamma\rangle+ \sum\limits_{(n,d)\neq (0,0)}Q^d\langle \frac{\phi_\alpha}{1-qL},\gamma,t,\ldots,t\rangle^{\epsilon,S_n}_{0,n+2,d}\right)g^{\alpha\beta}\phi_\beta
\]

The operator $(S^\epsilon_t)^*:\mathcal K\rightarrow \mathcal K$ is defined as
\[
(S^\epsilon_t)^*(q)(\gamma)=\sum_{\alpha,\beta}\left(\langle\gamma,\phi_\alpha\rangle+\sum\limits_{(n,d)\neq (0,0)}Q^d\langle \frac{\gamma}{1-qL},\phi_\alpha,t,\ldots,t\rangle^{\epsilon,S_n}_{0,n+2,d}\right)G^{\alpha\beta}_\epsilon\phi_\beta,
\]

\begin{proposition}\label{prop:unitary}
The operator $S^\epsilon_t$ is a unitary operator:
\[
(S_t^\epsilon)^*(q)\circ(S_t^{\epsilon})(1/q)=Id,
\]
for all $\epsilon \geq 0+$, where $0+$ represents a sufficiently small positive rational number such that $\epsilon=0+$ is in the first chamber $(0,\frac{1}{d(L_\theta)}]$ for all class $d$.
\end{proposition}
\begin{proof}
Consider elements $p_0,p_\infty\in K^0_{\mathbb C^*}(\mathbb P^1)$ defined by the restriction to the fixed points:
\[
p_0|_0=q, p_0|_\infty=1, \quad \text{and} \quad p_\infty|_0=1, p_\infty|_\infty=1/q.
\]
For arbitrary elements $\gamma, \delta \in K(W\sslash G)$, we consider the generating series
\begin{equation}\label{graph-space-gen-func}
\sum\limits_{n\geq 0, d\geq 0}Q^d\langle \gamma(1-p_0),t,\ldots,t,\delta(1-p_\infty)\rangle_{0,n+2,d}^{QG^\epsilon,S_n}
\end{equation}
Applying $\mathbb C^*$-localization, the generating series \eqref{graph-space-gen-func} can be written as
\begin{align*}
& \sum\limits_{\alpha,\beta} \left(\langle \phi_\alpha,\gamma \rangle+ \llangle \frac{\phi_\alpha}{1-qL},\gamma\rrangle_{0,2}^\epsilon \right)g^{\alpha\beta}\left(\langle \delta,\phi_\beta\rangle +\llangle \delta,\frac{\phi_\beta}{1-L/q}\rrangle_{0,2}^\epsilon\right) & \\
=&
\sum\limits_{\alpha,\beta}\langle \phi_\alpha,\gamma\rangle g^{\alpha\beta}
\left(\langle \delta,\phi_\beta\rangle+\llangle \delta,\phi_\beta \rrangle_{0,2}^\epsilon\right) +O\left(\frac{1}{1-q}\right)\\
=& \sum\limits_{\alpha}\langle \phi_\alpha,\gamma\rangle (\langle \delta,\phi^\alpha\rangle +\llangle \delta, \phi^\alpha\rrangle_{0,2}^\epsilon)+O\left(\frac{1}{1-q}\right)\\
=& 
\langle \delta,\gamma \rangle +\llangle \delta,\gamma \rrangle_{0,2}^\epsilon+O\left(\frac{1}{1-q}\right)
\end{align*}
where the first equation comes from the expansions
\[
\frac{1}{1-qL}=\sum_{i\geq 0} \frac{q^i(L-1)^i}{(1-q)^{i+1}}
\]
and 
\begin{align*}
\frac{1}{1-L/q}&=\sum_{i\geq 0} \frac{(1/q)^i(L-1)^i}{(1-1/q)^{i+1}}=\sum_{i\geq 0}\frac{q(L-1)^i}{(q-1)^{i+1}}\\
&=1+\frac{1}{q-1}+\sum_{i\geq 1}\frac{q(L-1)^i}{(q-1)^{i+1}}=1+O\left(\frac{1}{1-q}\right).
\end{align*}
On the other hand, the generating series \eqref{graph-space-gen-func} has no pole at $q=1$, hence 
\[
\sum\limits_{\alpha,\beta} \left(\langle \phi_\alpha,\gamma \rangle+ \llangle \frac{\phi_\alpha}{1-qL},\gamma\rrangle_{0,2}^\epsilon \right)g^{\alpha\beta}\left(\langle \delta,\phi_\beta\rangle +\llangle \delta,\frac{\phi_\beta}{1-L/q}\rrangle_{0,2}^\epsilon\right)= \langle \delta,\gamma \rangle+\llangle \delta, \gamma \rrangle_{0,2}^\epsilon.
\]
Therefore
\begin{align*}
& (S_t^\epsilon)^*(q)\circ(S_t^{\epsilon})(1/q)(\gamma)\\
=& 
\sum\limits_{\alpha,\beta}G^{\alpha\beta} \phi_\beta  
\sum\limits_{a,b} \left(\langle\phi_b,\phi_\alpha \rangle+\llangle \frac{\phi_b}{1-qL},\phi_\alpha\rrangle^{\epsilon}_{0,2}\right)
\left(\langle \phi_a,\gamma\rangle+ \llangle \frac{\phi_a}{1-L/q},\gamma\rrangle^{\epsilon}_{0,2}\right)g^{ab}\\
=&
\sum\limits_{\alpha,\beta}G^{\alpha\beta}\phi_\beta(\langle \phi_\alpha,\gamma\rangle+\llangle \phi_\alpha,\gamma\rrangle_{0,2}^\epsilon )\\
=&
\sum_\beta \phi_\beta \langle \phi^\beta,\gamma\rangle\\
=&
\gamma
\end{align*}
\end{proof}

\subsection{The $P$-series}
We also consider a generating series on the graph space known as $P$-series.
\[
P^\epsilon(t,q):=\sum_{\alpha,\beta}\phi_\beta G^{\alpha\beta}_\epsilon\llangle \phi_\alpha(1-p_\infty)\rrangle_{0,1}^{QG^\epsilon}
\]
where $p_\infty\in K^0_{\mathbb C^*}(\mathbb P^1)$ is defined by the restriction to the fixed points:
\[
p_\infty|_0=1, \quad p_\infty|_\infty=1/q.
\]
Then we have 
\begin{proposition}\label{prop:SP}
For every $\epsilon\geq 0+$, the equation
$\mathcal J_t^\epsilon(q)=S_t^\epsilon(q)(P^\epsilon(t,q))$ holds.
\end{proposition}
\begin{proof}
Apply $\mathbb C^*$-localization to $P^\epsilon(t,q)$, we have
\begin{align*}
 P^\epsilon(t,q)= 
 \sum_{\alpha,\beta}\phi_\beta G^{\alpha\beta}_\epsilon
 \sum_{a,b} &
\left( 
\phi_a+\frac{1}{1-q}\langle \phi_a,t \rangle+\llangle \frac{\phi_a}{(1-q)(1-qL)}\rrangle_{0,1}^\epsilon
\right)
g^{ab}\\
& \cdot \left(
\langle\phi_\alpha,\phi_b \rangle+\llangle (1-1/q)\phi_\alpha,\frac{\phi_b}{((1-1/q)(1-L/q)} \rrangle_{0,2}^\epsilon
\right)
\end{align*}
Hence
\begin{align*}
 P^\epsilon(t,q)=& 
 \sum_{\alpha,\beta}\phi_\beta G^{\alpha\beta}_\epsilon
\left(
\langle \phi_\alpha,\mathcal J^\epsilon\rangle +\llangle \phi_\alpha,\frac{\mathcal J^\epsilon}{1-L/q}\rrangle^\epsilon_{0,2}
\right)\\
=& (S_t^\epsilon)^*(1/q)(J^\epsilon)
\end{align*}
Then the proposition follows from the unitary property of the $S$-operators (Proposition \ref{prop:unitary}).
\end{proof}

Since we have the expansion
\[
J^\epsilon(t,q)=\mathds 1+O(Q)+O\left(\frac{1}{1-q}\right)
\]
and 
\begin{align*}
(S_t^\epsilon)^*(1/q)(\gamma)= & \sum_{\alpha,\beta}\left( \langle \gamma,\phi_\alpha\rangle+\llangle \gamma,\phi_\alpha\rrangle_{0,2}^\epsilon\right)G^{\alpha\beta}_\epsilon\phi_\beta+O\left(\frac{1}{1-q}\right)\\
=& \gamma+O\left(\frac{1}{1-q}\right)
\end{align*}
Hence
\[
(S_t^\epsilon)^*(1/q)(J^\epsilon(t,q))=\mathds 1+O(Q)+O\left(\frac{1}{1-q}\right)
\]
The fact that $P^\epsilon(t,q)$ has no pole at $q=1$ implies 
\[
P^\epsilon(t,q)=(S_t^\epsilon)^*(1/q)(J^\epsilon(t,q))=\mathds 1+O(Q)
\]
In particular, for $\epsilon>1$, we have  $J(t,q)=\mathds 1 +O\left(\frac{1}{1-q}\right)$ and $P(t,q)=\mathds 1$. Then Proposition \ref{prop:SP} becomes
\[
J(t,q)=S_t(q)(\mathds 1),
\]
which is a consequence of the $K$-theoretic string equation in \cite{Givental}.

Let $\gamma \in K^0(W\sslash G)\otimes \Lambda$ be an invertible element of the form $\gamma=\mathds 1+O(Q)$. For each $\epsilon \geq 0+$, we have
\[
S^\epsilon_t(q)(\gamma)=\gamma+\tau_\gamma^\epsilon(t)\frac{1}{1-q}+\left(\frac{1}{1-q}\right)^2 p_\epsilon\left(\frac{q}{1-q}\right),
\]
where $p_\epsilon\left(\frac{q}{1-q}\right)$ stands for a power series in the variable $q/(1-q)$ with coefficients in $K^0(W\sslash G)\otimes \Lambda \{\{t_i\}\}$. 
For $\gamma=\mathds 1 +O(Q)$, we have
\begin{align*}
S_t^\epsilon(\gamma)=& S_t^\infty(\gamma) & \mod Q\\
=& S_t^\infty(\mathds 1) & \mod Q\\
=& J^\infty(t) & \mod Q\\
=& \mathds 1+\frac{t}{1-q}+\left(\frac{1}{1-q}\right)^2p_\infty\left(\frac{q}{1-q}\right) & \mod Q,
\end{align*}
where $p_\infty\left(\frac{q}{1-q}\right)$ stands for a power series in $q/(1-q)$ with coefficients in $K^0(W\sslash G)\otimes \Lambda \{\{t_i\}\}$. 
Therefore
\[
\tau_\gamma^\epsilon(t)=t+O(Q).
\]
is an invertible transformation on $K^0(W\sslash G)\otimes \Lambda$. For $0+\leq \epsilon_1 \leq \epsilon_2 \leq \infty$, let
\[
\tau_{\gamma}^{\epsilon_1,\epsilon_2}(t):=(\tau_\gamma^{\epsilon_1})^{-1}\circ \tau_\gamma^{\epsilon_2}(t),
\]
then
\[
S_{\tau_{\gamma}^{\epsilon_1,\epsilon_2}(t)}^{\epsilon_1}(q)(\gamma)=S_t^{\epsilon_2}(q)(\gamma)+\frac{1}{(1-q)^2}p_{\epsilon_1,\epsilon_2}\left(\frac{q}{1-q}\right),
\]
where $p_{\epsilon_1.\epsilon_2}\left(\frac{q}{1-q}\right)$ stands for a power series in $q/(1-q)$ with coefficient in $K^0(W\sslash G)\otimes \Lambda \{\{t_i\}\}$. 
In general, we have the following lemma.

\begin{lemma}\label{lem:difference}
For every $\epsilon \geq 0+$, there exist a uniquely determined element $P^{\infty,\epsilon}(t,q)\in \mathcal K_+$, convergent in the $Q$-adic topology for each $t$, and a uniquely determined map $t\mapsto \tau^{\infty,\epsilon}(t)$ on $K^0(W\sslash G)\otimes \Lambda$ satisfying the following properties:
\begin{itemize}
\item $\tau^{\infty,\epsilon}(t)=t  \mod \, Q$;
\item $P^{\infty,\epsilon}(t,q)=\mathds 1 \mod \, Q$;
\item $ S^\epsilon_t(q)(P^\epsilon(t,q))=S^\infty_{\tau^{\infty,\epsilon}(t)}(q)(P^{\infty,\epsilon}(\tau^{\infty,\epsilon}(t),q))+ \frac{1}{(1-q)^2}p_{\epsilon,\infty}\left(\frac{q}{1-q}\right)$, where $p_{\epsilon,\infty}\left(\frac{q}{1-q}\right)$ stands for a power series in the variable $q/(1-q)$ with coefficient in $K^0(W\sslash G)\otimes \Lambda \{\{t_i\}\}$. 
\end{itemize}
\end{lemma}

\begin{proof}
Elements $\tau^{\infty,\epsilon}$ and $P^{\infty,\epsilon}(t,q)$ can be constructed by induction on the degree $d$ and the construction is unique. The construction is not difficult but messy, and is therefore omitted. 
\end{proof}

\begin{remark}\label{rmk:cone}
We write
\[
P^{\infty,\epsilon}(t,z)=\sum\limits_{i}C_i(Q,\{t_i\},q)\phi_i,
\]
where the coefficients $C_i(Q,\{t_i\},q)\in \Lambda[[\{t_i\},q,1/q]]$. Hence by general properties of $K$-theoretic Gromov-Witten invariants (see \cite{Givental}), the element
\[
S_t^\infty(q)(P^{\infty,\epsilon}(t,q))=\sum\limits_i C_i (1-q)\partial_{t_i} J^\infty (t,q)
\]
is on the Lagrangian cone $\mathcal L_{S_\infty,W\sslash G}$ associated to the genus $0$ $K$-theoretic Gromov-Witten theory of $W\sslash G$.
\end{remark}


\section{Genus $0$ Wall-Crossing}
The purpose of this section is to establish a wall-crossing result that relates genus $0$ $K$-theoretic quasimap invariants for different stability parameters. The proof, which is parallel to the treatment for cohomological quasimap theory in \cite{CK} and the toric case in \cite{Givental}, is based on localization. To this end, we assume in this section that there is a torus $T$ action on $W$ and the action commutes with the $G$-action. This induces $T$-actions on $[W/G]$ and $W\sslash G$. We assume that the $T$-action on $W\sslash G$ has isolated fixed points and isolated $1$-dimensional orbits.

\subsection{Fixed Point Localization}
 Following the analysis of fixed loci in \cite{CK}, we can describe the fixed loci of the $T$-equivariant quasimaps. 

Consider the $T$-equivariant version of the permutation-equivariant $K$-theoretic $S^\epsilon$-operator, denoted by $S^\epsilon_{t,T}$. Consider the $T$-fixed point basis $$\{\phi_\beta\}_\beta\subset K^0(W\sslash G)_T.$$ 
Given a $T$-fixed point $\beta\in (W\sslash G)^T$, we have the restriction of $S^\epsilon_{t,T}$ to $\beta$:
\[
\beta^*S^\epsilon_{t,T}(q)(\gamma)=\sum\limits_{\alpha}\left(\langle \phi_\alpha,\gamma\rangle+ \sum\limits_{(n,d)\neq (0,0)}Q^d\langle \frac{\phi_\alpha}{1-qL},\gamma,t,\ldots,t\rangle^{\epsilon,S_n,T}_{0,n+2,d}\right)g^{\alpha\beta}
\]

Fix $\epsilon>0$, an effective class $d\neq 0$ and a non-negative integer $n$. For a $T$-fixed quasimap
\[
((C,x_1,x_2,\ldots,x_{n+2}), P,\mu),
\]
and $(P^{\prime},\mu^{\prime})$, the restriction of the pair $(P,\mu)$ to an irreducible component $C^{\prime}$ of $C$, the rational map
\[
[\mu^{\prime}]: C^\prime \dashedrightarrow W\sslash G,
\]
induces a regular map
\[
[\mu^\prime]_{reg}:C^\prime \rightarrow W\sslash G.
\]
Then the regular map $[\mu^\prime]$ satisfies one of the following three conditions:
\begin{itemize}

\item $[\mu^\prime]_{reg}$ is a constant map and maps to a $T$-fixed point of $W\sslash G$, in this case, we call $C^\prime$ a contracted component;

\item there are no base points (i.e. $[\mu^\prime]_{reg}=[\mu^\prime]$) and it is a cover of a $1$-dimensional orbit of the $T$-action on $W\sslash G$, totally ramified over the two fixed points of the orbit;

\item $[\mu^\prime]_{reg}$ is a ramified cover of an $1$-dimensional orbit as in the second case, but $[\mu^\prime]$ has a base-point at one of the fixed point and a node at the other fixed point.

\end{itemize}

Let $M$ be a connected component of the fixed locus $Q^\epsilon_{0,n+2}(W\sslash G,d)^T$. Following \cite{CK}, $M$ is of {\em initial type} if the first marking is on a contracted irreducible component of the domain curve, of {\em recursion type} if the first marking is on a non-contracted irreducible component.

\begin{lemma}[Poles of $\beta^*S^\epsilon_{t,T}$]\label{lem:pole}
The restriction $\beta^*S^\epsilon_{t,T}$ is a rational function of $q$ with possibly poles only at $0$, $\infty$, roots of unity and at most simple poles at $q=(\lambda(\beta,\mu)^{-1/m})$, where $\lambda(\beta,\mu)$ is the character of the torus action on the tangent line at the fixed point $\mu$ corresponding to the $1$-dimensional orbit connecting the fixed points $\beta$ and $\mu$; for some $m=1,2,\ldots$.

\end{lemma}
\begin{proof}

We apply virtual localization to the sum 
\[
\sum_\alpha Q^d\langle \frac{\phi_\alpha}{1-qL},\gamma,t,\ldots,t\rangle^{\epsilon,S_n, T}_{0,n+1,d}g^{\alpha\beta}
\]
Note that we have the formal expansion
\[
\frac{1}{1-xL}=\sum\limits_{i\geq 0}\frac{x^i}{(1-x)^{i+1}}(L-1)^i.
\]

For each initial component $M$ with the first marking lying over $\beta$, we claim that it contributes polynomials of $\frac{1}{1-\xi q}$, where $\xi$ is a root of unity:

The vertex factor of the initial component $M$ that corresponds to the fixed point $\beta$ can be written as the fiber product
\[
\left(\left(\overline{Q}_{0,val(\beta)+k}(W\sslash G,0)^T\times_{(W\sslash G)^{T}}\overline{Q}_{0,1}(W\sslash G,d_1)_{\infty}^T\right)\times \cdots \right)\times_{(W\sslash G)^{T}}\overline{Q}_{0,1}(W\sslash G,d_k)_{\infty}^T,
\]
where the moduli space $\overline{Q}_{0,1}(W\sslash G,d_i)_{\infty}$ parametrizes the quasimaps $(\mathbb P^1,P,u)$ of class $d_i$ with a principal $G$-bundle $P$ on $\mathbb P^1$, a section $u:\mathbb P^1\rightarrow P\times_G W$ such that 
$u(x)\in W^s$ for $x\neq \infty \in \mathbb P^1$ and $\infty \in \mathbb P^1$ is a base-point of length $d_i(L_\theta)$.
We also have 
\[
\overline{Q}_{0,val(\beta)+k}(W\sslash G,0)^T=\overline{M}_{0,val(\beta)+k}
\]
is a finite dimensional complex manifold, hence $L$ restricts to a unipotent element and the trace $\on{tr}_h(\frac{1}{1-qL})=\frac{1}{1-q\xi \tilde L}$, where $h\in S_n$, $\xi$ is the eigenvalue of $h$ on $L$ and $\tilde L$ is the restriction of $L$ to the fixed point locus of $h$.

For each recursion component $M$ with the first marking lies over $\beta$, then the restriction of $L$ to $M$ is $\lambda(\beta,\mu)^{1/m}$. Hence $\beta^*S^\epsilon_{t,T}$ has simple poles at $q=\lambda(\beta,\mu)^{-1/m}$.
\end{proof}

\begin{lemma}[Recursion Relation]\label{lem:rec}
The restriction $\beta^*S^\epsilon_{t,T}$ of $S^\epsilon_{t,T}$ to the fixed point $\beta$ satisfies the recursion relation
\begin{equation}\label{recursion-relation}
\beta^*S^\epsilon_{t,T}(q)=I^\epsilon_\beta (q)+\sum_{\mu\in o(\beta)}\sum_{m=1}^{\infty} \frac{Q^{md(\beta,\mu)}}{m}\frac{\phi^\beta}{C_{\beta,\mu,m}}\frac{1}{1-\lambda(\beta,\mu)^{1/m}q} \mu^*S^\epsilon_{t,T}(\lambda(\beta,\mu)^{1/m})
\end{equation}
where 
\begin{enumerate}
\item
$I^\epsilon_\beta (q)$ is the sum of the contribution of all the components of initial type. Each $Q, \{t_i\}$ coefficient of $I^\epsilon_\beta(q)$ is of the form 
\[
\sum_{\xi: \text {root of unity}}\sum_{i\geq 0} c_{i,\xi}(\xi q)^i/(1-\xi q)^{i+1}.
\]
\item
$o(\beta)$ is the set of all fixed points $\mu$ connected to $\beta$ via a $1$-dimensional orbit, $d(\beta,\mu)$ is the homology class of the orbit and $\lambda(\beta,\mu)$ is the character of the torus representation on the tangent line at $\beta$ corresponding to the orbit. 
\item
The recursion coefficient $C_{\beta,\mu,m}$ is the $T$-equivariant $K$-theoretic Euler class of the virtual cotangent space to the moduli space $\overline{M}_{0,2}(W\sslash G,m)$ at the corresponding fixed point and this recursion coefficient does not depend on $\epsilon$.
\end{enumerate}
\end{lemma}
\begin{proof}
We again apply virtual localization to the summand 
\[
\sum_\alpha Q^d\langle \frac{\phi_\alpha}{1-qL},\gamma,t,\ldots,t\rangle^{\epsilon,S_n, T}_{0,n+1,d}g^{\alpha\beta}.
\]
Consider fixed components of recursion type,
there are two possibilities when the component has contributions to the pole at $q=\lambda(\beta,\mu)^{-1/m}$:

The first possibility is that $M$ is an one-dimensional orbit with multiplicity $m$ connecting the the fixed points $\beta, \mu$ and has a marked point at the other end. It is an isolated fixed point $p$ in $\overline{M}_{0,2}(W\sslash G, m)$. Note that this component does not dependent on the choose of $\epsilon$. The contribution is
\[
(\sum_\alpha \phi_\alpha g^{\alpha\beta})\frac{1}{1-\lambda(\beta,\mu)^{1/m}q}\frac{Q^{md(\beta,\mu)}}{m C_{\beta,\mu,m}}\sum_\nu \langle \phi_\nu,\gamma\rangle g^{\nu\mu}
\]
where the recursion coefficient 
\begin{equation}\label{rec-coeff}
C_{\beta,\mu,m}=\on{Euler}_T^K(T^*_p \overline{M}_{0,2}(W\sslash G,m))
\end{equation} 
is the $T$-equivariant $K$-theoretic Euler class of the virtual cotangent space to the moduli space at the point $p$ corresponding to the $m$-multiple cover of the $1$-dimensional orbit connecting fixed points $\beta$ and $\mu$. Hence $C_{\beta,\mu,m}$ does not dependent on the choose of $\epsilon$.

The second possibility is $M$ has an one-dimensional orbit connecting $\beta,\mu$ and a subgraph $M^\prime$ attached to $\mu$, i.e. the first marking of $M^\prime$ lies over $\mu$. The contribution is
\[
\frac{\sum_\alpha \phi_\alpha g^{\alpha\beta}}{1-\lambda(\beta,\mu)^{1/m}q}\frac{Q^{md(\beta,\mu)}}{mC_{\beta,\mu,m}}(\mu^* S^\epsilon_{t,T}(t,\lambda(\beta,\mu)^{-1/m})-\sum_\nu \langle \phi_\nu,\gamma\rangle g^{\nu\mu}).
\]
Finally, the polynomiality of the coefficients of $I^\epsilon_\beta(q)$ follows from unipotency of $L$.

This completes the proof.
\end{proof}

\begin{remark}
Write
\[
P^\epsilon(t,q)=P^\epsilon(t,(\lambda(\beta,\mu))^{-1/m})+(1-\lambda(\beta,\mu)^{1/m}q)A^\epsilon_{\mu,m}(q),
\]
where $A^\epsilon_{\mu,m}(q)$ is a power series in $(1-q)$. Applying localization as in the previous lemma, we have the following recursion relation for $S^\epsilon_t(q)(P^\epsilon(t,q))$:
\begin{align*}
& \beta^*S^\epsilon_t(q)(P^\epsilon(t,q))\\
& =\tilde{I}^\epsilon_\beta(q)+\sum_{\mu\in o(\beta)}\sum_{m=1}^{\infty} \frac{Q^{md(\beta,\mu)}}{m}\frac{\phi^\beta}{C_{\beta,\mu,m}}\frac{1}{1-\lambda(\beta,\mu)^{1/m}q} \mu^*S^\epsilon_{t,T}(\lambda(\beta,\mu)^{1/m})(P^\epsilon(t,q))\\
&= I^\epsilon_\beta(q)+\sum_{\mu\in o(\beta)}\sum_{m=1}^{\infty} \frac{Q^{md(\beta,\mu)}}{m}\frac{\phi^\beta}{C_{\beta,\mu,m}}\frac{1}{1-\lambda(\beta,\mu)^{1/m}q} \mu^*S^\epsilon_{t,T}(\lambda(\beta,\mu)^{1/m})(P^\epsilon(t,\lambda(\beta,\mu)^{-1/m}))
\end{align*}
where $\tilde{I}^\epsilon_\beta(q)$ is the summation of the contribution of all the components of initial type and
\[
I^\epsilon_\beta(q):=\tilde{I}^\epsilon_\beta(q)+\sum_{\mu\in o(\beta)}\sum_{m=1}^{\infty} \frac{Q^{md(\beta,\mu)}}{m}\frac{\phi^\beta}{C_{\beta,\mu,m}} \mu^*S^\epsilon_{t,T}(\lambda(\beta,\mu)^{1/m})(A^\epsilon_{\mu,m}(q))
\]
The recursion relation for $S^\infty_{\tau^{\infty,\epsilon}(t)}(q)(P^{\infty,\epsilon}(\tau^{\infty,\epsilon}(t),q))$ works in the same way.
\end{remark}

\subsection{Main Results}
This section, we state and prove the main theorems of this paper.
\begin{theorem}\label{main-theorem-1}
Assume the torus $T$ action on $W\sslash G$ has isolated fixed points and isolated $1$-dimensional orbits. Let $0+\leq \epsilon_1 < \epsilon_2\leq \infty$, $\gamma\in K^0_{T}(W\sslash G)\otimes \Lambda$ is of the form $\mathds 1+O(Q)$. Then 
\[
S^{\epsilon_1}_{\tau_{\gamma}^{\epsilon_1,\epsilon_2}(t)}(q)(\gamma)=S^{\epsilon_2}_t(q)(\gamma).
\]
\end{theorem}
\begin{theorem}\label{main-theorem-2}
Assume the torus $T$ action on $W\sslash G$ has isolated fixed points and isolated $1$-dimensional orbits, then for all $\epsilon \geq 0+$,
\[
\mathcal J^\epsilon(t,q)=S^\epsilon_t(q)(P^\epsilon(t,q))=S^\infty_{\tau^{\infty,\epsilon}(t)}(q)(P^{\infty,\epsilon}(\tau^{\infty,\epsilon}(t),q)),
\]
hence, lies on $\mathcal L_{S_\infty,W\sslash G}$, the Lagrangian cone of the permutation-equivariant  $K$-theoretic Gromov-Witten theory of $W\sslash G$.
\end{theorem}

Theorem \ref{main-theorem-2} is a consequence of Theorem \ref{main-theorem-1} and the arguments in Remark \ref{rmk:cone}.

\begin{remark}
Certainly, it is conjectured that the statements in Theorems \ref{main-theorem-1} and \ref{main-theorem-2} hold for $W\sslash G$ without torus actions.
\end{remark}

\begin{lemma}\label{lem:poly}
For each torus fixed point $\beta\in (W\sslash G)^T$, the series
\[
D(S^\epsilon_\beta):=S^\epsilon_\beta(Q,t,q)S^\epsilon_\beta(Q(1/q)^{aL_\theta},t,1/q)
\]
has no pole at roots of unity, where $a$ varies in integers and $(Q(1/q)^{aL_\theta})^d=Q^d (1/q)^{a d(L_\theta)}$.
\end{lemma}
\begin{proof}

Given a fixed point $\beta\in (W\sslash G)^T$, we write $QG^\epsilon_{0,m+2,d}(W\sslash G)_\beta$ for the $T$-fixed locus parametrizes quasimaps with the parametrized $\mathbb P^1$ contracted to the point $\beta$. Let $\kappa$ be the inclusion map into the graph space and put
\[
\on{Res}_\mu\mathcal O_{QG^\epsilon_{0,m+2,d}(W\sslash G)}^{\text{vir}}:=\kappa_*\left( \frac{\mathcal O_{QG^\epsilon_{0,m+2,d}(W\sslash G)_\beta}^{\text{vir}}}{\on{Euler}^K_T(N^{\text{vir}})}\right)
\]
the $T$-equivariant residue at the fixed locus $QG^\epsilon_{0,m+2,d}(W\sslash G)_\beta$. Furthermore, we write
\[
\gamma=\sum_d Q^d\gamma_d
\]
where $\gamma_d\in K^0_{T}(W\sslash G)\otimes \mathbb Q$. Consider the generating series
\begin{multline*}
  \llangle \gamma(1-p_0),\gamma(1-p_\infty);U(L_\theta)\rrangle_{0,2;\beta}^{QG,\epsilon}:=\\
\sum_{m,d,d_1,d_2}Q^d Q^{d_1}Q^{d_2}
\chi\left(QG^\epsilon_{0,m+2,d}(W\sslash G)_\beta/S_m, \right.
\on{Res}_\mu\mathcal O_{QG^\epsilon_{0,m+2,d}(W\sslash G)}^{\text{vir}} \\
  \otimes 
\left. (U_{d_1,d_2}(L_\theta))^a\on{ev}_1^*(\gamma_{d_1}(1-p_0))\on{ev}_2^*(\gamma_{d_2}(1-p_\infty))\prod_{i=1}^m \on{ev}_{i+2}^*(t)\right)
\end{multline*}
where, $U_{d_1,d_2}(L_\theta)$ is the universal $\mathbb C^*$-equivaraint line bundle obtained from pulling back $\mathcal O(1)$ with the canonical linearization, as described in \cite[Section 3.3]{CK}. This is defined without $\mathbb C^*$-localization, hence has no pole at $q=1$. Applying $\mathbb{C}^*$-localization to this generating series, we have
\[
\llangle \gamma(1-p_0),\gamma(1-p_\infty);U(L_\theta)\rrangle_{0,2;\beta}^{QG,\epsilon}=
\lambda(\mathcal O(\theta);\beta)^a D(S^\epsilon_\beta),
\]
where $\lambda(\mathcal O(\theta);\beta)$ is the character of the torus representation on the fiber of $\mathcal O(\theta)$ at $\beta$. 

For any positive integer $m$, $\Psi^m(D(S^\epsilon_\beta))$ is also defined without $\mathbb C^*$-localization, hence has no pole at $q=1$, where the Adams operation extended from $\Lambda$ by $\Psi^m(q)=q^m$. Therefore, $D(S^\epsilon_\beta)$ has no pole at the $m$-th root of unity.
Hence the Lemma follows.
\end{proof}

\begin{remark}
Since $P^\infty$ and $P^{\infty,\epsilon}$ in Theorem \ref{main-theorem-2} have no pole at roots of unity, the proof also applies to the restrictions of $S^\epsilon_t(q)(P^\epsilon(t,q))$ and $S^\infty_{\tau^{\infty,\epsilon}(t)}(q)(P^{\infty,\epsilon}(\tau^{\infty,\epsilon}(t),q))$ to fixed points, after appropriate adjustments to the generating series.
\end{remark}

\begin{lemma}[Uniqueness Lemma]
Let 
\[
\{S_{1,\beta}\}_{\beta\in (W\sslash G)^T} \quad \text{and} \quad \{S_{2,\beta}\}_{\beta\in (W\sslash G)^T}
\]
be two systems of power series in $\Lambda [[t_i]]\{\{q,1/q\}\}$ that satisfy the following properties:
\begin{description}
\item[(1)]  For all $\beta\in (W\sslash G)^T$, $S_{1,\beta}$ and $S_{2,\beta}$ are rational functions of $q$ with possibly poles only at $0$, $\infty$, roots of unity and at most simple poles at $q=(\lambda(\beta,\mu)^{-1/m})$, where $\lambda(\beta,\mu)$ is the character of the torus representation on the tangent line at the fixed point $\beta$ corresponding to the $1$-dimensional orbit connecting the fixed points $\beta$ and $\mu$, for $m=1,2,\ldots$.

\item[(2)] The systems
\[
\{S_{1,\beta}\}_{\beta\in (W\sslash G)^T} \quad \text{and} \quad \{S_{2,\beta}\}_{\beta\in (W\sslash G)^T}
\]
both satisfy the recursion relation \eqref{recursion-relation}.

\item[(3)] For all $\beta\in (W\sslash G)^T$, the series 
\[
D(S^\epsilon_{1,\beta}):=S^\epsilon_{1,\beta}(Q,t,q)S^\epsilon_{1,\beta}(Q(1/q)^{aL_\theta},t,1/q)
\]
and
\[
D(S^\epsilon_{2,\beta}):=S^\epsilon_{2,\beta}(Q,t,q)S^\epsilon_{2,\beta}(Q(1/q)^{aL_\theta},t,1/q)
\]
have no pole at $q=1$.

\item[(4)] For all $\beta\in (W\sslash G)^T$, 
\[
S_{1,\beta}=S_{2,\beta}+\frac{1}{(1-q)^2}p_{1,2}\left(\frac{q}{1-q}\right),
\]
where $p_{1,2}\left(\frac{q}{1-q}\right)$ stands for a power series in the variable $q/(1-q)$ with coefficient in $K^0(W\sslash G)\otimes \Lambda \{\{t_i\}\}$. 
\item[(5)] For all $\beta\in (W\sslash G)^T$,
\[
S_{1,\beta}=S_{2,\beta} \mod \, Q
\]
\end{description}
Then $S_{1,\beta}=S_{2,\beta}$ for all $\beta\in (W\sslash G)^T$.
\end{lemma}
\begin{proof}
We write
\[
S_{1,\beta}=\sum_d Q^d\sum_{k}c_{1,\beta,d,k}(q)\prod_i t_i^{k_i} \quad \text{and} \quad
S_{2,\beta}=\sum_d Q^d\sum_{k}c_{2,\beta,d,k}(q)\prod_i t_i^{k_i}
\]
where $k:=\{k_i\}_i$ and $k_i$ are nonnegative integers. We define the bi-degree of the monomial $Q^d\prod_i t_i^{k_i}$ to be 
$\left( \sum_i k_i,d(L_\theta)\right)$. We write $S_{1,\beta}^{(m,l)}$ and $S_{2,\beta}^{(m,l)}$ for the part of bi-degree $(m,l)$ of $S_{1,\beta}$ and $S_{2,\beta}$, respectively. It suffices to show
\[
S_{1,\beta}^{(m,l)}=S_{2,\beta}^{(m,l)},\quad \text{for all }\beta\in (W\sslash G)^T \text{ and } (m,l)\in \mathbb N\times \mathbb N
\]
We prove it by induction on $(m,l)$ using the lexicographic order
\[
(m^\prime,l^\prime)<(m,l) \text{ if and only if } m^\prime <m, \text{ or } m^\prime =m \text{ and } l^\prime <l.
\]
The base case when $d=0$ is true due to property (5). 

For $l\geq 1$, we assume 
\[
S_{1,\beta}^{(m^\prime,l^\prime)}=S_{2,\beta}^{(m^\prime,l^\prime)} \text{ for all } \beta\in (W\sslash G)^T \text{ and all } (m^\prime,l^\prime)<(m,l)
\]
Denote by $D^{(m,l)}$ the part of bi-degree (m,l) of the difference $D(S_{1,\beta})-D(S_{2,\beta})$. By induction, we have
\[
D^{(m,l)}=S_{1,\beta}^{(m,l)}(q)-S_{2,\beta}^{(m,l)}(q)+ (1/q)^{al}(S_{1,\beta}^{(m,l)}(1/q)-S_{2,\beta}^{(m,l)}(1/q)).
\]
By properties (2), $\left( S_{1,\beta}^{(m,l)}(q)-S_{2,\beta}^{(m,l)}(q)\right)$ is  the sum of monomials of the form $c_i(\xi q)^i/(1-\xi q)^{i+1}$, for $i\geq 0$ and roots of unity $\xi$, with coefficient in $K^0(W\sslash G)\otimes \Lambda\{\{t_i\}\}$, we write
\[
\left(S_{1,\beta}^{(m,l)}(q)-S_{2,\beta}^{(m,l)}(q)\right)_1
\] 
for the sum of terms of $S_{1,\beta}^{(m,l)}(q)-S_{2,\beta}^{(m,l)}(q)$ with $\xi=1$,
therefore, by property (4), we have
\[
\left(S_{1,\beta}^{(m,l)}(q)-S_{2,\beta}^{(m,l)}(q)\right)_1=\left(\frac{1}{1-q}\right)^n(Aq^{n-2}+O(1-q))
\]
and
\[
\left(S_{1,\beta}^{(m,l)}(1/q)-S_{2,\beta}^{(m,l)}(1/q)\right)_1=\left(\frac{1}{1-q}\right)^n((-1)^nAq^2+O(1-q)),
\]
for an integer $n\geq 2$ and a nonzero element $A\in \Lambda\{\{t_i\}\}$. Moreover, we have
\[
(1/q)^{al}=1+al(1-q)+O(1-q),
\]
therefore
\[
\left(D^{(m,l)}\right)_1=\left(\frac{1}{1-q}\right)^n(Aq^{n-2}+(-1)^nAq^2+O(1-q))
\]

For $n>1$, then $D^{(m,l)}$ has a pole at $q=1$. It contradicts property (3). Therefore, $S_{1,\beta}^{(m,l)}-S_{2,\beta}^{(m,l)}$ has no pole at $q=1$. Similar argument shows $S_{1,\beta}^{(m,l)}-S_{2,\beta}^{(m,l)}$ has no pole at roots of unity. Therefore, $S_{1,\beta}^{(m,l)}=S_{2,\beta}^{(m,l)}$.
\end{proof}

Theorem \ref{main-theorem-1} now follows from the above uniqueness lemma applied to $\{S^{\epsilon_1}_{\tau_{\gamma}^{\epsilon_1,\epsilon_2}(t), \beta}\}$ and $\{S^{\epsilon_2}_{t, \beta}\}$. The required properties are checked before. Property (1) is in Lemma \ref{lem:pole}. Property (2) is in Lemma \ref{lem:rec}. Property (3) is in Lemma \ref{lem:poly}. Property (4) is in Lemma \ref{lem:difference}. Property (5) is clear.

\end{document}